\documentclass{amsart}
\usepackage{amsmath}
\usepackage{amsfonts}
\usepackage{amssymb}
\usepackage{graphicx}
\usepackage{color}
\usepackage[latin1]{inputenc}
\usepackage{float}
\usepackage[font=footnotesize,labelfont=bf]{caption}
\usepackage{wrapfig}

\newtheorem{theorem}{Theorem}
\newtheorem{lemma}[theorem]{Lemma}
\newtheorem{prop}[theorem]{Proposition}
\newtheorem{obs}[theorem]{Observation}
\newtheorem{rem}[theorem]{Remark}
\newtheorem{cor}[theorem]{Corollary}

\newtheorem{conj}{Conjecture}

\def\id{\rm{Id}\,}
\def\diam{\rm{diam}\,}

\begin{document}
\title{Convex and weakly convex domination in prism graphs}
\author{Monika Rosicka}
\address{Faculty of Mathematics, Physics and Informatics, University of Gda\'nsk,
80-952 Gda\'nsk, Poland; Institute of Theoretical Physics and
Astrophysics, and National Quantum Information Centre in Gda\'nsk, 81-824
Sopot, Poland.}

 \begin{abstract}
For a given graph $G=(V,E)$ and permutation $\pi:V\mapsto V$ the prism $\pi G$ of $G$ is defined as follows: $V(\pi G)=V(G)\cup V(G')$, where $G'$ is a copy of $G$, and $E(\pi G)=E(G)\cup E(G')\cup M_{\pi}$, where $M_{\pi}=\{uv': u\in V(G), v=\pi(u)\}$ and $v'$ denotes the copy of $v$ in $G'$. 

We study and compare the properties of convex and weakly convex dominating sets in prism graphs. In particular, we characterize prism $\gamma_{con}$-fixers and -doublers. We also show that the differences $\gamma_{wcon}(G)-\gamma_{wcon}(\pi G)$ and $\gamma_{wcon}(\pi G) - 2\gamma_{wcon}(G)$ can be arbitrarily large, and that the convex domination number of $\pi G$ cannot be bounded in terms of $\gamma_{con}(G).$
 \end{abstract} 

\maketitle


\section{Introduction}



Let $G=(V_G,E_G)$ be an undirected graph, let $\pi\colon V_G\mapsto V_G$ be a permutation of its vertex set and let $G'$ be a vertex-disjoint copy of $G$. We denote the copy of~a~vertex $v\in V_G$ in $G'$ by $v'$. If $S$ is a set of vertices of $G$, then $S'$ denotes the~copy of $S$ in $G'$, i.e. the set $\{v':v\in S\}.$ By $V$ and $V'$ we denote the vertex sets of $G$ and $G'$, respectively.
The \textit{prism graph} $\pi G$ is a graph with the vertex set $V\cup V'$ and the edge set $E_G\cup E_{G'}\cup M_{\pi},$ where $M_{\pi}=\{u\pi(u)': u\in V_G\}.$

We say that a set $A\in V_G$ dominates $B\in V_G$ if $B\in N_G[A].$ We denote this by $A\succ B.$ A set $A\in V_G$ is called a \emph{dominating set} of the graph $G$ if it dominates $V_G,$ i.e. if $N_G[A]=V_G.$ The minimum cardinality of a dominating set in $G$ is called the \emph{domination number} of $G$ and denoted $\gamma(G).$ A minimum dominating set of a graph is sometimes called a $\gamma$-\emph{set}.

Prism graphs were first defined in \cite{ChH} and the problem of domination in prism graphs was first studied by Burger, Mynhardt and Weakley \cite{BMW}.

In a connected graph $G$ a set of vertices $A\subseteq V_G$ is said to be \emph{convex} if for every pair of vertices $u,v\in A$ the set $A$ contains all vertices of every shortest $u$--$v$ path.
A \emph{convex dominating set} of $G$ is a dominating set of~$G$ which is convex.
The minimum cardinality of a convex dominating set of $G$ is called the \emph{convex domination number} of~$G$ and denoted as $\gamma_{con}(G)$.
A $\gamma_{con}$-set of~$G$ is a convex dominating set of cardinality $\gamma_{con}(G)$.

 A set $A$ of vertices in a connected graph $G$ is said to be \emph{weakly convex} if for every pair of vertices $u,v\in A$ it contains all vertices of at least one shortest $u$--$v$.
A \emph{weakly convex dominating set} $G$ is a dominating set of $G$ which is weakly convex.
The minimum cardinality of a weakly convex dominating set of $G$ is called the \emph{weakly convex domination number} and denoted as $\gamma_{wcon}(G)$.
A $\gamma_{wcon}$-set of $G$ is a minimum weakly convex dominating set.

A set $A\subseteq V_G$ is said to be \emph{connected} if it induces a connected subgraph, i.e., if for every pair of vertices $u,v\in A$ there exists a $u$--$v$ path contained entirely in $A$.
A \emph{connected dominating set} of $G$ is a dominating set which is connected.
The \emph{connected domination number} of $G$, denoted as $\gamma_c(G)$, is the minimum cardinality of a connected dominating set of $G$.
A $\gamma_c$-set of $G$ is a~connected dominating set of cardinality $\gamma_c(G).$

In this paper we compare the properties of convex and weakly convex sets in prism graphs, particularly convex and weakly convex dominating sets. We also generalize some known properties of convex domination in prism graphs to weakly convex and connected domination.

\section{Connected, convex and weakly convex domination}

It is clear that the notions of convex, weakly convex and connected domination are closely related. Since every convex set is weakly convex and every weakly convex set is connected, it is easy to see that $\gamma(G)\leq\gamma_c(G)\leq\gamma_{wcon}(G)\leq\gamma_{con}(G)$. Many, but not all, properties of connected sets can be extended to both convex and weakly convex sets.

\begin{lemma}
\label{Lplus1}
If $G$ is a connected graph with and $\pi$ is a permutation of $V_G$, then $\gamma_{c}(\pi G)\geq\gamma(G)+1.$
\end{lemma}
\begin{proof}
Obviously, if $\gamma(\pi G)\geq\gamma(G)+1$, then $\gamma_{c}(\pi G)\geq\gamma(G)+1.$

If $\gamma(\pi G)=\gamma(G)$ and $D$ is a $\gamma$-set of $\pi G$, we denote $D_1=D\cap V$ and $D_2'=D\cap V'.$  
Since the only vertices in $V'$ dominated by $D_1$ are in $\pi(D_1),$ it follows that $B=\pi(D_1)\cup D_2\succ V$ and therefore $\left|B\right|=\left|D\right|=\gamma(G).$ Thus, $\pi(D_1)\cap D_2=\emptyset.$ Since $N_{V'}(D_1)=\pi(D_1)',$ it follows that $E(D_1,D_2')=\emptyset$ and therefore $D$ is not connected. Thus, every connected dominating set of $\pi G$ has cardinality at least $\gamma(G)+1.$
\end{proof}

Since $\gamma_c(G)\leq\gamma_{wcon}(G)\leq\gamma_{con}(G),$ Lemma \ref{Lplus1} also applies to convex and weakly convex domination, that is, for any graph $G$ and permutation $\pi$ of~$V_G$ we have $\gamma_{con}(\pi G)\geq\gamma(G)+1$ and $\gamma_{wcon}(\pi G)\geq\gamma(G)+1.$

\begin{obs}
\label{obs:connected}
For any connected graph $G$ and permutation $\pi$ of~$V_G$
let $A=A_1\cup A_2\cup A_3$ be a dominating set of $G$ such that:
\begin{enumerate}
\item $A_1\cup A_2\succ V_G-A_3$,
\item $A_1\cup A_2$ is connected,
\item $\pi(A_2\cup A_3)$ is connected,
\item $\pi(A_2\cup A_3)\succ V_G-\pi(A_1).$
\end{enumerate}
Then $D=A_1\cup A_2\cup (\pi(A_2\cup A_3))'$ is a connected dominating set of size $\left|A\right|+\left|A_2\right|$ in $\pi G$.
\end{obs}
\begin{proof}
Since $A_1\cup A_2\succ (\pi(A_1))'\cup V-A_3$ and $(\pi(A_2\cup A_3))'\succ A_3\cup V'-(\pi(A_1))'$, it is clear that $D=A_1\cup A_2\cup(\pi(A_2\cup A_3))'\succ ((\pi(A_1))'\cup V-A_3)\cup (A_3\cup V'-(\pi(A_1))')=V\cup V',$ that is, $D$ is a dominating set in $\pi G$. The sets $D_1=A_1\cup A_2$ and $D_2'=(\pi(A_2\cup A_3))'$ are connected and $E(D_1,D_2')\neq\emptyset,$ it follows that $D$ is connected. Hence $D$ is a connected dominating set of $\pi G.$
\end{proof}

Note that this particular result cannot be extended to convex domination as, for example, if $\pi = \id,$ then for any set $A=A_1\cup A_2\cup A_3$, where $A_1,A_2$ and $A_3$ are not empty, the set $A_1\cup A_2\cup (A_2\cup A_3)'$ is not convex. For connected domination, however, we can use it to show exactly when the bound from Lemma \ref{Lplus1} is achieved.

\begin{prop}
\label{Lconnected}
Let $G$ be a connected graph with $\gamma_{con}(G)\neq\left|V_G\right|$ and let $\pi$ be a permutation of $V_G.$
Then $\gamma_c(\pi G)=\gamma(G)+1$ if and only if there exists a $\gamma$-set $A=A_1\cup A_2$ of $G$ and a vertex $v\in A_1$ such that:
\begin{enumerate}
\item $A_1\succ V_G-A_2,$
\item $A_1$ is connected,
\item $\pi(A_2\cup \{v\})\succ V_G-\pi(A_1),$
\item $\pi(A_2\cup \{v\})$ is connected.
\end{enumerate}
\end{prop}
\begin{proof}
Let $D=A_1\cup(\pi(A_2\cup \{v\}))'$ be a set of vertices in $\pi G.$
It follows from Observation \ref{obs:connected} that $D$ is a connected dominating set of size $\gamma(G)+1$ in $\pi G$, so $\gamma_c(\pi G)\leq\gamma(G)+1$. By Lemma \ref{Lplus1}, $\gamma_c(\pi(G))\geq\gamma(G)+1$. Thus, $\gamma_c(\pi G)=\gamma(G)+1.$

Now let $D$ be a connected dominating set of size $\gamma(G)+1$ in $\pi G$ and let $D_1=D\cap V$ and $D_2'=D\cap V'.$ 

Since $D$ dominates $V\cup V',$ it is necessary that $D_1\succ V-(\pi^{-1}(D_2)-\{v\})$ and $D_2\succ V-\pi(D_1).$ 

Now let $A_1=D_1$ and $A_2=\pi^{-1}(D_2)-D_1$. If $\left|A\right|>\gamma(G),$ then $\left|D\right|>\gamma(G)+1$, a contradiction. Thus, $A$ is a $\gamma$-set, $A_1\succ V_G-A_2$ and $\pi(A_2\cup \{v\})\succ V_G-\pi(A_1).$

Since $D$ is connected, there exists a vertex $v\in D_1$ such that $\pi(v)\in D_2.$

Suppose $\left|D_1\cap D_2\right|>1.$ Then $A=D_1\cup\pi^{-1}(D_2)$ is a dominating set and $\left|A\right|<\left|D\right|-1.$ This implies $\left|D\right|>\gamma(G)+1.$

Finally, if $A_1$ is not connected, then there exists a vertex $v_i$ in each of its components such that $\pi(v_i)\in \pi(A_2\cup \{v\})$, as otherwise the set $D$ would not be connected. Then $\left|D\right|>\gamma(G)+1$. Similarly, if $\pi(A_2\cup\pi(v))$ is not connected, then $\left|D\right|>\gamma(G)+1$.
\end{proof}

Again, the above is not true for convex domination, however, in Section \ref{sec:idG} we prove a related property of weakly convex dominating sets.

\section{Generalizing some properties of convex domination}

Convex domination in prism graphs was studied by Lema\'nska and Zuazua in~\cite{LZ}, where they prove the following theorem.

\begin{theorem} [\cite{LZ}]
\label{Tconvexfixer}
Let $G$ be a connected graph. If \\$\gamma_{con}(G)=\left|V_G\right|$ and $\diam G\leq 2,$ then $\gamma_{con}(G)=\gamma_{con}(\pi G)$ for every permutation $\pi$ of $V_G$.
\end{theorem}

The proof of Theorem \ref{Tconvexfixer} relies some properties of convex dominating sets, which they prove in the same paper. We will now compare these properties with those of weakly convex dominating sets.

\begin{theorem} [\cite{LZ}]
\label{Tdiam2V}
For any connected graph $G$:
\begin{enumerate}
\item If $\diam(G)\leq 2,$ then $V$ and $V'$ are both convex dominating sets of $\pi G$ for any permutation $\pi.$
\item If $\diam(G)\geq 3,$ then there exist permutations $\pi_1$ and $\pi_2$ such that $V$ is not a convex dominating set of $\pi_1G$ and $V'$ is not a convex dominating set of $\pi_2G.$
\end{enumerate}
\end{theorem}

It follows that if $\diam(G)\leq 2,$ then $\gamma_{con}(\pi G)\leq\left|V_G\right|$ for any $\pi$. The weakly convex domination number has a similar property.

\begin{theorem}
\label{Tdiam3V}
Every connected graph $G$ has the following properties:
\begin{enumerate}
\item If $\diam(G)\leq 3,$ then $V$ and $V'$ are weakly convex dominating sets of $\pi G$ for every permutation $\pi;$
\item If $\diam(G)>3$, then there exist permutations $\pi_1$ and $\pi_2$ such that $V$ is not a weakly convex dominating set of $\pi_1G$ and $V'$ is not a weakly convex dominating set of $\pi_2G.$
\end{enumerate}
\end{theorem}
\begin{proof}
Obviously, for any permutation $\pi,$ $V$ and $V'$ are dominating sets of~$\pi G.$ For any pair of vertices $u,v\in V,$ the shortest $u$--$v$ path containing at least one vertex from $V'$ has length at least $3.$ Thus, if $\diam(G)\leq 3,$ then $d_G(u,v)=d_{\pi G}(u,v)$ and $V$ contains at least one shortest $u$--$v$ path in $\pi G.$ Similarly, if $\diam(G)\leq 3,$ then $V'$ must contain a shortest $u'$--$v'$ path in $\pi G$ for every $u',v'\in V'.$

Now, let $\diam(G)$ be at least $4$ and let $u,v\in V$ be a pair of vertices such that $d_G(u,v)\geq 4.$ If $\pi_1$ is a permutation such that $\pi_1(u)$ and $\pi_1(v)$ are adjacent, then $d_{\pi G}(u,v)=3<d_G(u,v)$. Thus, $V$ is not a weakly convex dominating set in $\pi_1G.$ Similarly, for $\pi_2=\pi_1^{-1}$, the set $V'$ is not weakly convex.
\end{proof}

The first part of Theorem \ref{Tdiam3V} implies that $\gamma_{wcon}(\pi G)\leq \left|V_G\right|$ for any graph $G$ with diameter at most $3.$

\begin{prop}[\cite{LZ}]
\label{PD1D2}
For a connected graph $G$ and permutation $\pi$ of $V_G$, let $D$ be a convex dominating set of $\pi G$ and let $D_1=D\cap V$ and $D_2'=D\cap V'.$ Then $D$ has the following properties:
\begin{enumerate}
\item If $\left|D\right|<\left|V_G\right|$, then $D_1\neq\emptyset$ and $D_2'\neq\emptyset;$
\item If $D_1\neq\emptyset$ and $D_2'\neq\emptyset,$ then there exists at least one $x\in D_1$ such that $\pi(x)\in D_2$.
\end{enumerate}
\end{prop}

The above can be generalized to weakly convex and connected domination.

\begin{prop}
\label{PD1D2connected}
For a connected graph $G$ and permutation $\pi$ of $V_G$, let $D$ be a connected dominating set of $\pi G$ and let $D_1=D\cap V$ and $D_2'=D\cap V'.$ Then $D$ has the following properties:
\begin{enumerate}
\item If $\left|D\right|<\left|V_G\right|$, then $D_1\neq\emptyset$ and $D_2'\neq\emptyset;$
\item If $D_1\neq\emptyset$ and $D_2'\neq\emptyset,$ then there exists at least one $x\in D_1$ such that $\pi(x)\in D_2$.
\end{enumerate}
\end{prop}

\begin{proof}
If $D_1=\emptyset,$ then $D=D_2'$. But every vertex in $D_2'$ only dominates one vertex in $V.$ It follows that if $D$ is a dominating, then $D_1=\emptyset$ implies $\left|D\right|=\left|V\right|$. The same reasoning applies if $D_2'=\emptyset.$ Thus, if $\left|D\right|<\left|V\right|,$ then $\emptyset\notin\{D_1,D_2'\}.$

If both sets $D_1$ and $D_2$ are nonempty, then $D$ contains a pair of vertices $v_1\in D_1, v_2'\in D_2'.$ Since the set $D$ is connected, there exists a $v_1$-$v_2'$ path contained entirely in $D$. Each vertex $v\in V$ has only one neighbor in $V',$ namely $\pi(v)'.$ Hence, every path connecting $v_1\in D_1$ with $v_2'\in D_2'$ contains a pair of vertices $x\in V$ and $\pi(x)'\in V'.$ This shows that $D_1$ contains a vertex $x$ such that $\pi(x)\in D_2.$ 
\end{proof}

Since every weakly convex dominating set is a connected dominating set, we also have the following.

\begin{cor}
\label{CD1D2weak}
For a connected graph $G$ and permutation $\pi$ of $V_G$, let $D$ be a connected dominating set of $\pi G$ and let $D_1=D\cap V$ and $D_2'=D\cap V'.$ Then $D$ has the following properties:
\begin{enumerate}
\item If $\left|D\right|<\left|V_G\right|$, then $D_1\neq\emptyset$ and $D_2'\neq\emptyset;$
\item If $D_1\neq\emptyset$ and $D_2'\neq\emptyset,$ then there exists at least one $x\in D_1$ such that $\pi(x)\in D_2$.
\end{enumerate}
\end{cor}

\begin{lemma}[\cite{LZ}]
\label{LD1D2}
Let $G$ be a connected graph in which $\diam(G)\leq 2$. Let $D$ be a convex dominating set of $\pi G$ and let $D_1=D\cap V$ and $D_2'=D\cap V'.$ Then the set $D$ has the following properties:
\begin{enumerate}
\item If $\pi(D_1)\subseteq D_2$, then $D_2$ is a convex dominating set of $G;$
\item If $\pi^{-1}(D_2)\subseteq D_1$, then $D_1$ is a convex dominating set of $G.$
\end{enumerate}
\end{lemma}

Again, we can prove a similar property for weakly convex domination.

\begin{lemma}
\label{LD1D2}
Let $G$ be a connected graph in which $\diam(G)\leq 2$. Let $D$ be a weakly convex dominating set of $\pi G$ and let $D_1=D\cap V$ and $D_2'=D\cap V'.$ Then the set $D$ has the following properties:
\begin{enumerate}
\item If $\pi(D_1)\subseteq D_2$, then $D_2$ is a weakly convex dominating set of $G;$
\item If $\pi^{-1}(D_2)\subseteq D_1$, then $D_1$ is a weakly convex dominating set of $G.$
\end{enumerate}

\end{lemma}

\begin{proof}
If $\pi(D_1)\in D_2$, then clearly $D_2'$ dominates $V',$ as $D_1$ does not dominate any part of $V'-D_2'.$ It follows that $D_2$ is a dominating set of $G$. 
For any two vertices $u',v'\in D_2'$ the shortest possible $u'$-$v'$-path in $\pi G$ containing a vertex from $V$ has length at least $3$. If $\diam(G)\leq 2,$ then any shortest $u'$-$v'$-path must be contained in $D_2'$. Thus, for every $u,v\in D_2$ the set $D_2$ contains a shortest $u$-$v$-path. Hence, $D_2$ is a weakly convex dominating set of $G$.

Similarly, if $\pi^{-1}(D_2)\subseteq D_1$, then $D_1$ is a dominating set of $G$ because $D_2'$ does not dominate any part of $V-D_1$. $D_1$ is also convex, because for any $u,v\in D_1$ no shortest $u$-$v$-path passes through $D_2'.$
\end{proof}

Note that, unlike Theorem \ref{Tdiam3V}, the above is does not hold for every $G$ such that $\diam(G)\leq 3.$ For example, if $G=P_4=(\{1,2,3,4\},\{12,23,34\})$ and $\pi=(12)(34),$ the set $D=\{1',2,3,4'\}$ is a weakly convex dominating set of $\pi G$. The set $\{1,4\}$ contains $\pi(\{2,3\})$ and it is not a weakly convex set.

Since so many of the results in \cite{LZ} can be generalized to weakly convex domination, one could expect a weakly convex analogue of Theorem \ref{Tconvexfixer} to be true as well. However, Lema\'nska and Zuazua's proof relies on showing that if $G$ is a conncected graph with $diam(G)\leq 2$, then a dominating set of $\pi G$ with less than $\left|V_G\right|$ vertices cannot be convex. This is based on properties of convex sets which weakly convex sets do not have. Thus there appears to be no reason why such a set cannot be weakly convex and therefore the following conjecture seems more likely.

\begin{conj}
There exists a connected graph $G$ with $diam(G) = 2$ and permutation $\pi$ of $V_G$ such that $\gamma_{wcon}(G)=\left|V_G\right|$ and $\gamma_{wcon}(\pi G)<\gamma_{wcon}(G).$
\end{conj}

\section{Convex and weakly convex domination in $\id G$}
\label{sec:idG}

We now consider the special case where the permutation $\pi=\id.$ A graph $G$ is called a \emph{prism fixer} if $\gamma(\id G)=\gamma(G)$ and a \emph{prism doubler} if $\gamma(\id G)=2\gamma(G).$ A~graph $G$ is called a \emph{universal fixer} if $\gamma(\pi G)=\gamma(G)$ for every permutation $\pi$ of $V_G$ and a \emph{universal doubler} if $\gamma(\pi G)=2\gamma(G)$ for every $\pi.$ Prism fixers are characterized in \cite{MX} and universal fixers in \cite{MR} and \cite{KW}. Prism doublers and universal doublers are studied in \cite{BMW}.

Similarly, a graph $G$ such that $\gamma_{con}(\id G)=\gamma_{con}(G)$ is called a \emph{prism $\gamma_{con}$-fixer} and a graph with $\gamma_{con}(\id G)=2\gamma_{con}(G)$ is called a \emph{prism $\gamma_{con}$-doubler}. A universal $\gamma_{con}$-fixer is a graph such that for every $\pi$ $\gamma_{con}(\pi G)=\gamma_{con}(G)$ and a universal $\gamma_{con}$-doubler is a graph such that $\gamma_{con}(\pi G)=2\gamma_{con}(G)$ for every $\pi.$

We begin this section by studying some properties of convex and weakly convex sets in $\id G.$

\begin{lemma}
\label{L1wcon}
If $S\subset V_G$ is a weakly convex set in $G$, then $S,S'$ and $S\cup S'$ are weakly convex sets in $\id G.$
\end{lemma}

\begin{proof}
Let $u$ and $v$ be any two vertices in $S$. Since the set $S$ is weakly convex, it contains a shortest $u$--$v$ path in $G$. It is easy to see that $d_{\id G}(u,v)=d_G(u,v)$. Thus $S$ also contains a shortest $u$--$v$ path in $\id G$. Since the same applies to $u',v'\in S',$ the set $S'$ is also weakly convex. 

Now let $u=v_0, v_1,...,v_i,..., v_k=v$ be a shortest $u$--$v$ path in $G$ and let $v_i,...,v_{k-1}\in S$. The path $u=v_0, v_1,...,v_i,$ $v_i',v_{i+1}',..., v_k'=v'$ is a shortest $u$--$v'$ path in $\id G$ contained in $S\cup S'$. Thus $S\cup S'$ is a weakly convex set in $\id G.$
\end{proof}

The same applies to convex sets.

\begin{lemma}
\label{L1con}
If $S\subset V_G$ is a convex set in $G$, then $S,S'$ and $S\cup S'$ are convex sets in $\id G.$
\end{lemma}

\begin{proof}
Let $u$ and $v$ be any two vertices in $S$. Since the set $S$ is convex, it contains every $u$--$v$ path of length $d_G(u,v)$ in $G$. The shortest $u$--$v$ path in $\id G$ containing a vertex from $V'$ has length $d_G(u,v)+2.$ Thus, $d_{\id G}(u,v)=d_G(u,v)$ and $S$ contains every shortest $u$--$v$ path in $\id G.$ Since the same applies to $u',v'\in S',$ the set $S'$ is also convex. 

For two vertices $u\in S, v'\in S'$ we have $d_{\id G}(u,v')=d_G(u,v)+1.$ Every shortest $u$--$v'$ path in $\id G$ has the form $u=v_0, v_1,...,v_i,v_i',v_{i+1}',..., v_k'=v'$ for some shortest $u$--$v$ path $u=v_0, v_1,...,v_i,..., v_k=v$ in $G$. Since $S$ contains that $u$--$v$ path, $S\cup S'$ is a convex set in $\id G.$
\end{proof}

Weakly convex sets also have the following property.

\begin{lemma}
\label{L2wcon}
A set $S\in V_{\id G}$ is weakly convex if and only if $S_1=S\cap V, S_2'=S\cap V'$ and:

\begin{enumerate}
\item $S_1$, $S_2$ and $S_1\cup S_2$ are weakly convex sets in $G$
\item For every $u\in S_1, v\in S_2$, the shortest $u$-$v$ path in $S_1\cup S_2$ contains a vertex from $S_1\cap S_2$
\end{enumerate}

\end{lemma}

\begin{proof}
Let $S$ be a weakly convex set in $\id G$ and let $u,v\in S_1.$ Since there is no shortest $u$-$v$ path in $\id G$ containing a vertex from $V'$, the set $S_1$ must contain a shortest $u$-$v$ path, which is also a shortest $u$-$v$ path in $G$. The same applies to $u_1', v_1'\in S_2'.$ Thus $S_1$ and $S_2$ are weakly convex sets in $G$. 
Now let $u\in S_1$ and $v'\in S_2.$ If $u=v_0, v_1,...,v_i,..., v_k=v$ is a $u$-$v$ path in $G$ shorter than the shortest $u$-$v$ path in $S_1\cup S_2,$ then $u=v_0, v_1,...,v_i,v_i',v_{i+1}',..., v_k'=v'$ is a $u$-$v'$ path in $\id G$ shorter than the shortest $u$-$v'$ path in $S$ and $S$ is not weakly convex. Thus, $S_1\cup S_2$ is a weakly convex set in $G$. Furthermore, since $v_i\in S_1$ and $v_i'\in S_2',$ it is clear that $v_i\in S_1\cap S_2.$

Now let $S$ be a set satisfying conditions (1)-(2). By Lemma \ref{L1wcon} $S_1,S_2'$ and $S_1\cup S_2$ are convex in $\id G.$
Let $u\in S_1, v\in S_2.$ The set $S_1\cup S_2$ contains a shortest $u$-$v$ path $u=v_0,...,v_i,...,v_k=v,$ where $v_i\in S_1\cap S_2.$ Since $S_1$ and $S_2$ are weakly convex, the paths $v_0,...v_i$ and $v_i,...,v_k$ are contained in $S_1$ and $S_2$, respectively. Thus $u=v_0,...,v_i,v_i'...,v_k'=v'$ is a shortest $u$-$v'$ path in $\id G$ and $S$ contains all of its vertices. 
\end{proof}

This is not the case with convex sets. If a convex set in $\id G$ contains a pair of vertices $u,v'$ it must also contain $u'$ and $v.$ This leads to some differences between convex and weakly convex domination.

\begin{theorem}
\label{TidG}
If $G$ is any connected graph, then $\gamma_{con}(\id G)=\min\{2\gamma_{con}(G),\left|V_G\right|\}.$
\end{theorem}
\begin{proof}

Let $A$ be a $\gamma_{con}$-set of $G$. Obviously, $A\cup A'$ dominates $\id G$. By Lemma \ref{L1con}, $A\cup A'$ is a convex set in $\id G$. Hence, $A\cup A'$ is clearly a convex dominating set of $\pi G$. The set $V$ is also a convex dominating set of $\pi G.$

Let $D$ be a $\gamma_{con}$-set in $\pi G.$ If $D\cap V'=\emptyset,$ then obviously $D=V.$ Similarly, if $D\cap V=\emptyset,$ then $D=V'$.
If $D$ contains vertices $u\in V$ and $v'\in V',$ it must also contain every shortest $u$-$v'$ path in $\pi G$. Thus, $D$ contains $u',v$ and all shortest $u$-$v$ and $u'$-$v'$ paths. Therefore, in this case, $D=A\cup A'$ for some $A\subseteq V$ and $A$ is a $\gamma_{con}$-set of $G$.
\end{proof}

As a result, we have the following.

\begin{cor}
\label{Tconprismfix}
Every connected graph $G$ has the following properties:
\begin{enumerate}
\item $G$ is a prism $\gamma_{con}$-fixer if and only if $\gamma_{con}(G)=\left|V_G\right|;$
\item $G$ is a prism $\gamma_{con}$-doubler if and only if $\gamma_{con}(G)\leq\frac{1}{2}\left|V_G\right|.$
\end{enumerate}
\end{cor}
\begin{proof}
If $\gamma_{con}(\id G)=$$\gamma_{con}(G),$ then, obviously, $\gamma_{con}(\id G)\neq 2$$\gamma_{con}(G)$. Thus, by Theorem \ref{TidG}, we have $\gamma_{con}(\id G)=\left|V_G\right|.$

The set $V$ is a convex dominating set of $\id G.$ Thus, if $\gamma_{con}(G)=\left|V_G\right|,$ then, by Theorem \ref{TidG}, $V$ is a minimum convex dominating set of $\id G$.

By Theorem \ref{TidG}, $\gamma_{con}(\id G)=2$$\gamma_{con}(G)$ if and only if $2\gamma_{con}(G)\leq\left|V_G\right|.$ Thus, $\gamma_{con}(\id G)=2$$\gamma_{con}(G)$ if and only if $\gamma_{con}(G)=\frac{1}{2}\left|V_G\right|$.
\end{proof}

Since every universal $\gamma_{con}$-fixer is a prism $\gamma_{con}$-fixer, and every universal $\gamma_{con}$-doubler is a prism $\gamma_{con}$-doubler, we also have the following corollary.

\begin{cor}
\label{corconfix}
Let $G$ be a connected graph. Then:
\begin{enumerate}
\item If $G$ is a universal $\gamma_{con}$-fixer, then $\gamma_{con}(G)=\left|V_G\right|;$
\item If $G$ is a universal $\gamma_{con}$-doubler, then $\gamma_{con}(G)\leq\frac{1}{2}\left|V_G\right|.$
\end{enumerate}
\end{cor}

Weakly convex domination has a somewhat similar property.

\begin{theorem}
If $G$ is a connected graph, then $\gamma_{wcon}(\id G)\leq\min\{\left|V_G\right|, 2\gamma_{wcon}(G)\}$.
\end{theorem}

\begin{proof}
Obviously, $V$ is a dominating set in $\id G.$ By Lemma \ref{L1wcon}, it is also a weakly convex set in $\id G.$ Thus, $\gamma_{wcon}(\id G)\leq\left|V_G\right|.$

If $S$ is a $\gamma_{wcon}$-set of $G,$ then $S\cup S'$ is a dominating set in $\id G,$ as $S\succ V$ and $S'\succ V'.$ Since Lemma \ref{L1wcon} implies that $S\cup S'$ is a (not necessarily minimal) weakly convex dominating set in $\id G$ and thus $\gamma_{wcon}(\id G)\leq 2\gamma_{wcon}(G)$.
\end{proof}

However, thanks to Lemma \ref{L2wcon}, $\gamma_{wcon}(\id G)$ is not necessarily equal to $\min\{\left|V_G\right|,$ $2\gamma_{wcon}(G)\}$. In fact, the following is true.

\begin{theorem}
Let $G$ be a connected graph. The graph $\id G$ has a weakly convex dominating set $D\notin\{V,V'\}$ of cardinality $\gamma_{wcon}(G)+k$ if and only if $G$ has a weakly convex dominating set $A$ which can be partitioned into three nonempty sets $A_1,A_2,A_3$ such that $\left|A\right| + \left|A_2\right|=\gamma_{wcon}(G) + k$ and:

\begin{enumerate}
\item $A_1\cup A_2$ and $A_2\cup A_3$ are weakly convex,
\item For every $u\in A_1\cup A_2, v\in A_2\cup A_3$, the shortest $u$--$v$ path in $A$ contains a vertex from $A_2$,
\item $A_1\cup A_2\succ V-A_3$ and $A_3\cup A_2\succ V-A_1$.
\end{enumerate}

In particular, $\gamma_{wcon}(\id G) = \gamma_{wcon}(G) + 1$ if and only if $G$ has a $\gamma_{wcon}$-set $A=A_1\cup A_2\cup A_3$ such that conditions (1)-(3) are fulfilled and $\left|A_2\right| = 1$.

\end{theorem}

\begin{proof}
Let $A=A_1\cup A_2\cup A_3$ be a weakly convex dominating set of a connected graph $G$. The set $D=A_1\cup A_2\cup A_2'\cup A_3'$ is a dominating set of $\id G,$ as $A_1\cup A_2\succ V-A_3\cup A_1'$ and $A_2'\cup A_3'\succ V'-A_1'\cup A_3.$ By Lemma \ref{L2wcon}, $D$ is a weakly convex set. Thus, $\id G$ has a weakly convex dominating set $D\notin\{V,V'\}$ of size $\left|A\right|+\left|A_2\right|.$

Now let $D$ be a weakly convex set of $\id G$ and let $D_1=D\cap V$ and $D_2'=D\cap V'$. If either $D_1$ or $D_2'$ was an empty set, then $D$ would be either $V$ or $V'.$ By Lemma \ref{L2wcon} the sets $D_1$ and $D_2$ are convex and for every $u\in D_1, v\in D_2$ there is a shortest $u$-$v$ path in $D_1\cup D_2$ containing a vertex from $D_1\cap D_2$. Since the only veritces in $V'$ dominated by $D_1$ are in $D_1'$, it is clear that $D_2\succ V-D_1.$ Similarly, $D_1\succ V-D_2$, as the only vertices in $V$ dominated by $D_2'$ are in $D_2.$ Thus $A=D_1\cup D_2$ is a weakly convex domiating set of $G$ such that, for $A_1=D_1-D_2,$ $A_2=D_1\cap D_2$ and $A_3=D_2-D_1$ conditions $(1)$-$(3)$ are fulfilled. 
\end{proof}

For example graph $G$ in Fig. \ref{fig:Tw16} has such a $\gamma_{wcon}$-set. As a result $\gamma_{wcon}(\id G)<\min\{\left|V_G\right|,$ $2\gamma_{wcon}(G)\}$.

\begin{center}
\begin{figure}[H]
	\begin{picture}(300,170)
	\put(25,-10){\includegraphics[scale=0.35]{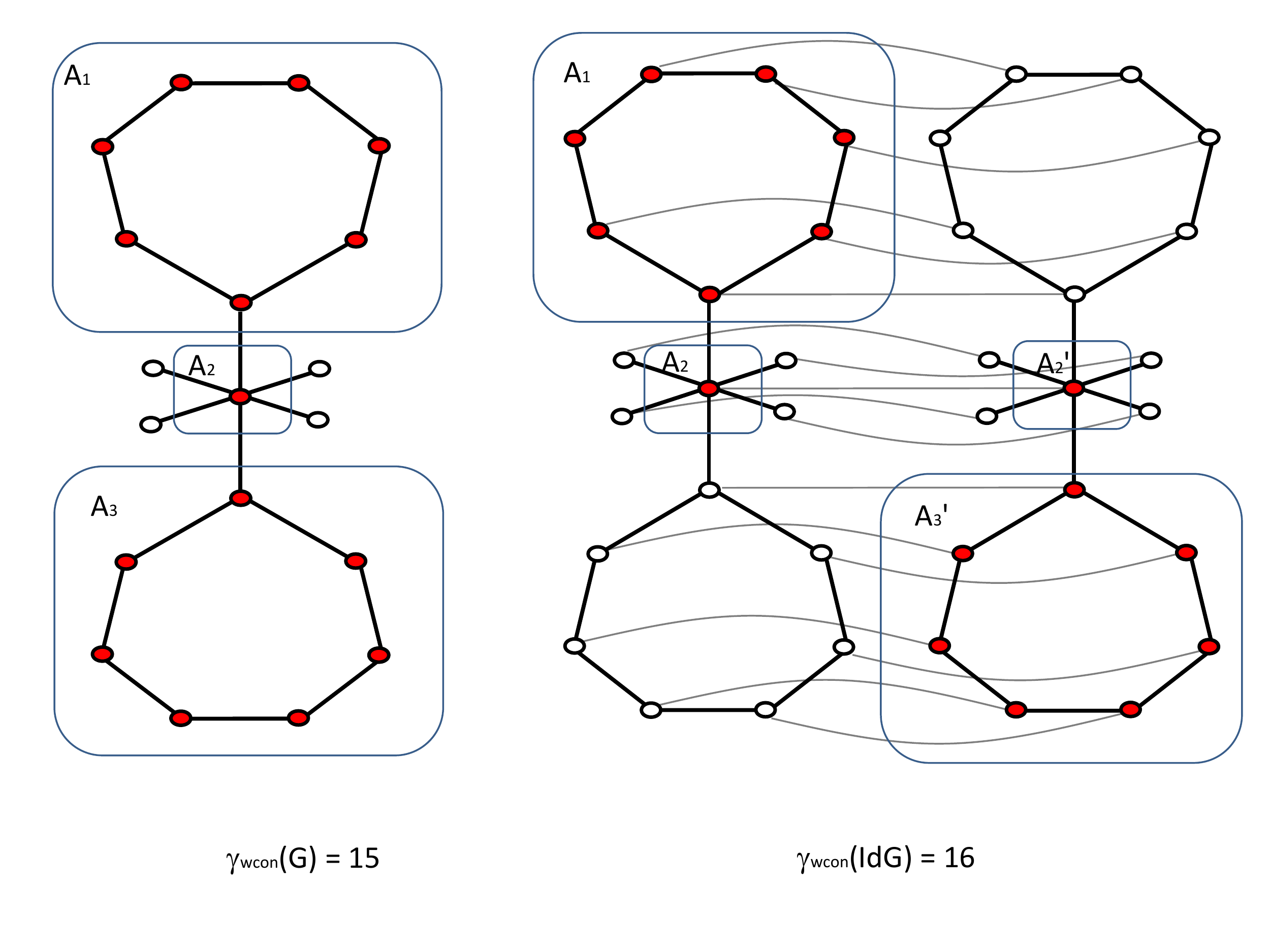}}
	\end{picture}
	\caption{The set $A_1\cup A_2\cup A_2'\cup A_3'$ is a weakly convex dominating set of $\id G$ with $\gamma_{wcon}(G)+1$ vertices.}
	\label{fig:Tw16}
\end{figure}
\end{center}

\section{Upper and lower bounds}

It is well known that the inequalities 

\begin{center}
\begin{equation}
\label{eq0}
 \gamma (G)\leq \gamma (\pi G)\leq 2\gamma (G)
\end{equation}
\end{center}

\noindent hold for any graph $G$ and any permutation $\pi$ of its vertex set. 
At the conference Colorings Inependence and Domination in 2015 Rita Zuazua conjectured that similar inequalities hold for convex and weakly convex domination, i.e.

\begin{center}
\begin{equation}
\label{eq2}
 \gamma_{wcon}(G)\leq\gamma_{wcon}(\pi G)\leq 2\gamma_{wcon}(G)
\end{equation}
\\and 

\begin{equation}
\label{eq1}
 \gamma_{con}(G)\leq\gamma_{con}(\pi G)\leq 2\gamma_{con}(G)
\end{equation}
\end{center}
\noindent However, this is not true in general. The smallest counterexample is the path $P_3$ with $V_{P_3}=\{1,2,3\}, E_{P_3}=\{12,23\}$ and the permutation $\pi=(12)$. In this case $\gamma_{con}(P_3) = \gamma_{wcon}(P_3) = 1$ while $\gamma_{con}(\pi P_3) =\gamma_{wcon}(P_3) = 3$.

For a star $K_{1,k}$ with $k\geq 2$ and the permutation $\pi=(01)$, where $0$ is the central vertex and $1$ is one of the other vertices, we have $\gamma_{con}(K_{1,k}) = \gamma_{wcon}(K_{1,k}) =1$ and $\gamma_{con}(\pi K_{1,k}) = 4$ while $\gamma_{wcon}(K_{1,k}) = 3$. Thus, the upper bounds in (\ref{eq2}) and (\ref{eq1}) not hold for $K_{1,k}.$

Furthermore, for every $k\in \mathbb{N}$ there is a graph $G$ and permutation $\pi$ such that $\gamma_{wcon}(\pi G)-2\gamma_{wcon}(G)\geq k$.

Let us begin with the cycle $C_7=(\{0,1,2,3,4,5,6\},\{01,12,23,34,45,56,60\})$ and the permutation $\pi=(13)(46).$ 
The weakly convex domination number of $C_7$ is $7$, but the graph $\pi C_7$ can be dominated by a weakly convex set with only $6$ vertices: $\{0,0',1,1',6,6'\}$.

In fact, the difference can be arbitrarily large. For any $k\in \mathbb{N}$ we can construct a graph $G_k$ as follows (see Fig. \ref{fig:C7ks}).

\begin{enumerate}
\item Take $k$ copies of $C_7$. Denote the $i$-th copy of the vertex $j$ by $(i,j),$
\item Replace the vertices $(1,0),...,(k,0)$ with a single vertex $(0,0).$
\end{enumerate}

The permutation $\pi_k$ is defined as $\pi_k(i,j)=(i,\pi(j)).$ Then $\gamma_{wcon}(G_k)=6k+1$ and $\gamma_{wcon}(\pi_kG_k)=4k+2$ (The set $\{(0,0),(0,0)',(1,1),...,(k,1),(1,1)',...,(k,1)',$\\$(1,6),...,(k,6),(1,6)',...,(k,6)'\}$ is a weakly convex domiating set of $\pi_kG_k$). Hence $\gamma_{wcon}(G_k)-\gamma_{wcon}(\pi_kG_k) = 2k-1.$ 

\begin{center}
\begin{figure}[H]
	\begin{picture}(300,190)
	\put(20,-20){\includegraphics[scale=0.4]{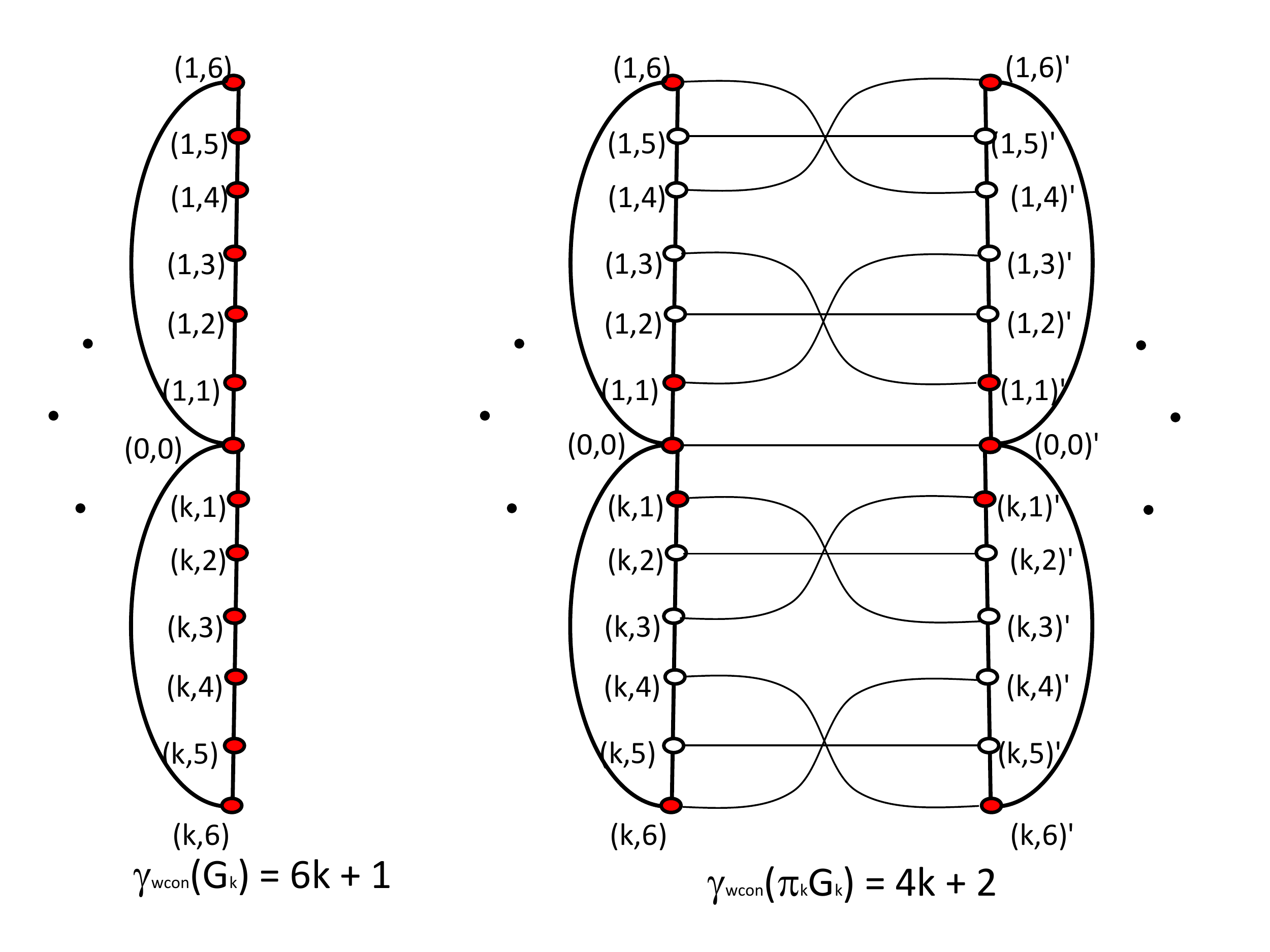}}
	\end{picture}
	\caption{The graphs $G_k$ and $\pi_kG_k$ and their $\gamma_{con}$-sets.}
	\label{fig:C7ks}
\end{figure}
\end{center}

The first inequality in (\ref{eq2}) can also be violated.

Let us consider the path $P_6=(\{0,1,2,3,4,5\},\{01,12,23,34,45\})$ and the permutation $\sigma=(14)(23).$ The weakly convex domination number of $P_6$ is $4$, but the weakly convex domination number of $\sigma P_6$ is $12$.

For $k\geq 2$ we construct the graph $H_k$ as follows

\begin{enumerate}
\item Take $k$ paths $P_6$. Denote the $i$-th copy of the vertex $j$ as $(i,j)$,
\item Replace the vertices $(1,0),...,(k,0)$ with a single vertex $(0,0).$
\end{enumerate}

\begin{center}
\begin{figure}[H]
	\begin{picture}(300,190)
	\put(25,-10){\includegraphics[scale=0.35]{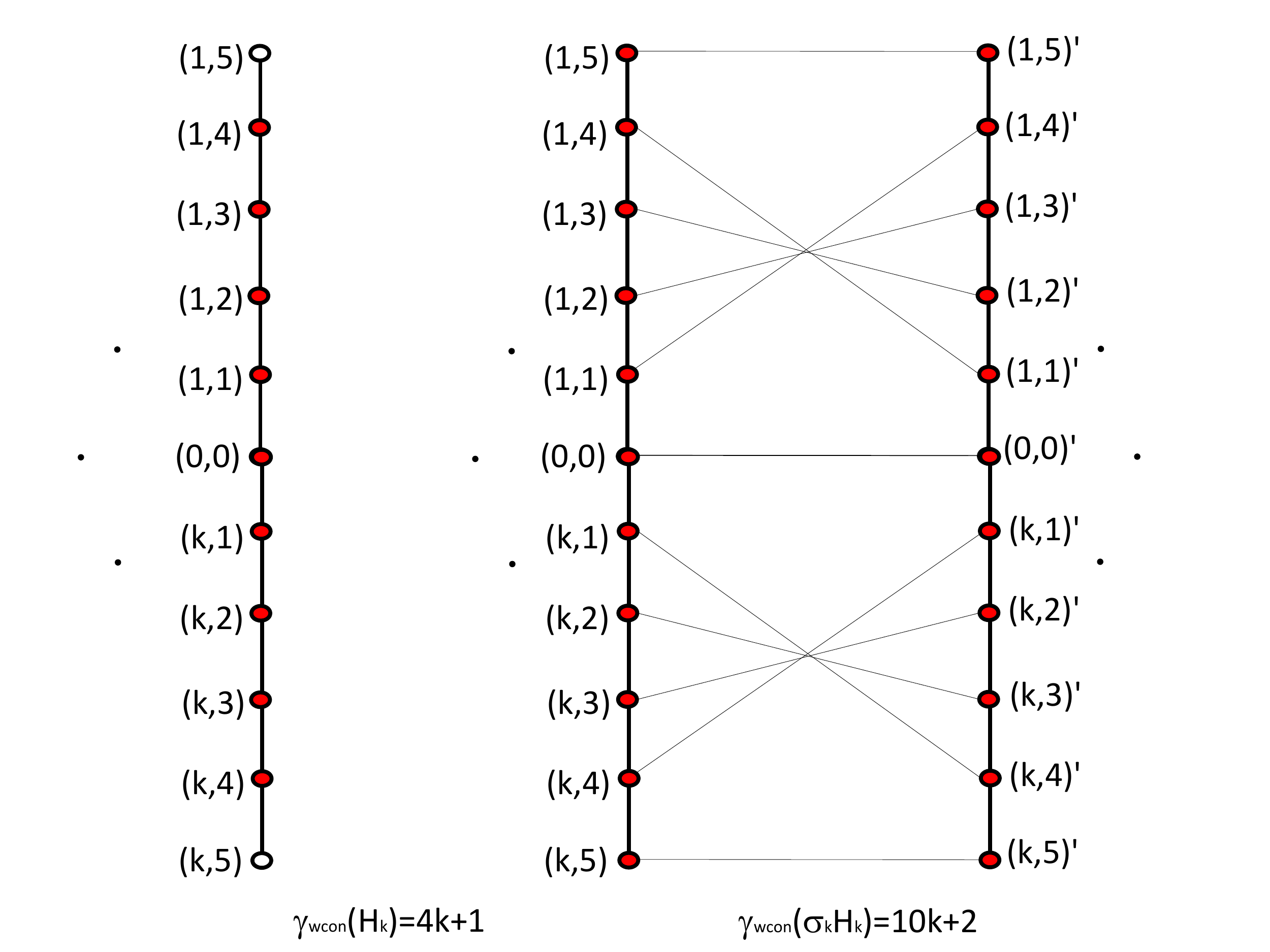}}
	\end{picture}
	\caption{The graphs $H_k$ and $\sigma_k H_k$ and their $\gamma_{wcon}$-sets.}
	\label{fig:p6}
\end{figure}
\end{center}

The permutation $\sigma_k$ is defined as $\sigma_k(i,j)=(i,\sigma(j)).$

It is easy to see that $\gamma_{wcon}(H_k)=4k+1$ and $\gamma_{wcon}(\pi_k H_k)=10k+2.$ Thus $\gamma_{wcon}(\sigma_k H_k)-2\gamma_{wcon}(H_k)=2k.$ Once again, the difference can be arbitrarily large.

Thus for any $k\in \mathbb{N}$ there exist graphs $G, H$ and permutation $\pi:V_G\mapsto V_G, \sigma:V_H\mapsto V_H$ such that $\gamma_{wcon}(G)-\gamma_{wcon}(\pi G)\geq k$ and $\gamma_{wcon}(\sigma H)-2\gamma_{wcon}(H)\geq k.$

Both inequalities (\ref{eq1}) are also violated by entire families of graphs.

Let $T_{k,l}$ be a tree with $V_{T_{k,l}}=\{0,1,...,k,(1,1),...,(1,l),...,(k,1),...,(k,l)\}$ and $E_{T_{k,l}}=\{0i:1\leq i\leq k\}\cup\{i(i,j): 1\leq i\leq k, 1\leq j \leq l\}$ for $k\geq 2$ and $l\geq 1$ (see Fig. \ref{fig:Tkl}) and let $\pi_{k,l}=(1...k).$ 

\begin{center}
\begin{figure}[H]
	\begin{picture}(300,190)
	\put(25,-10){\includegraphics[scale=0.4]{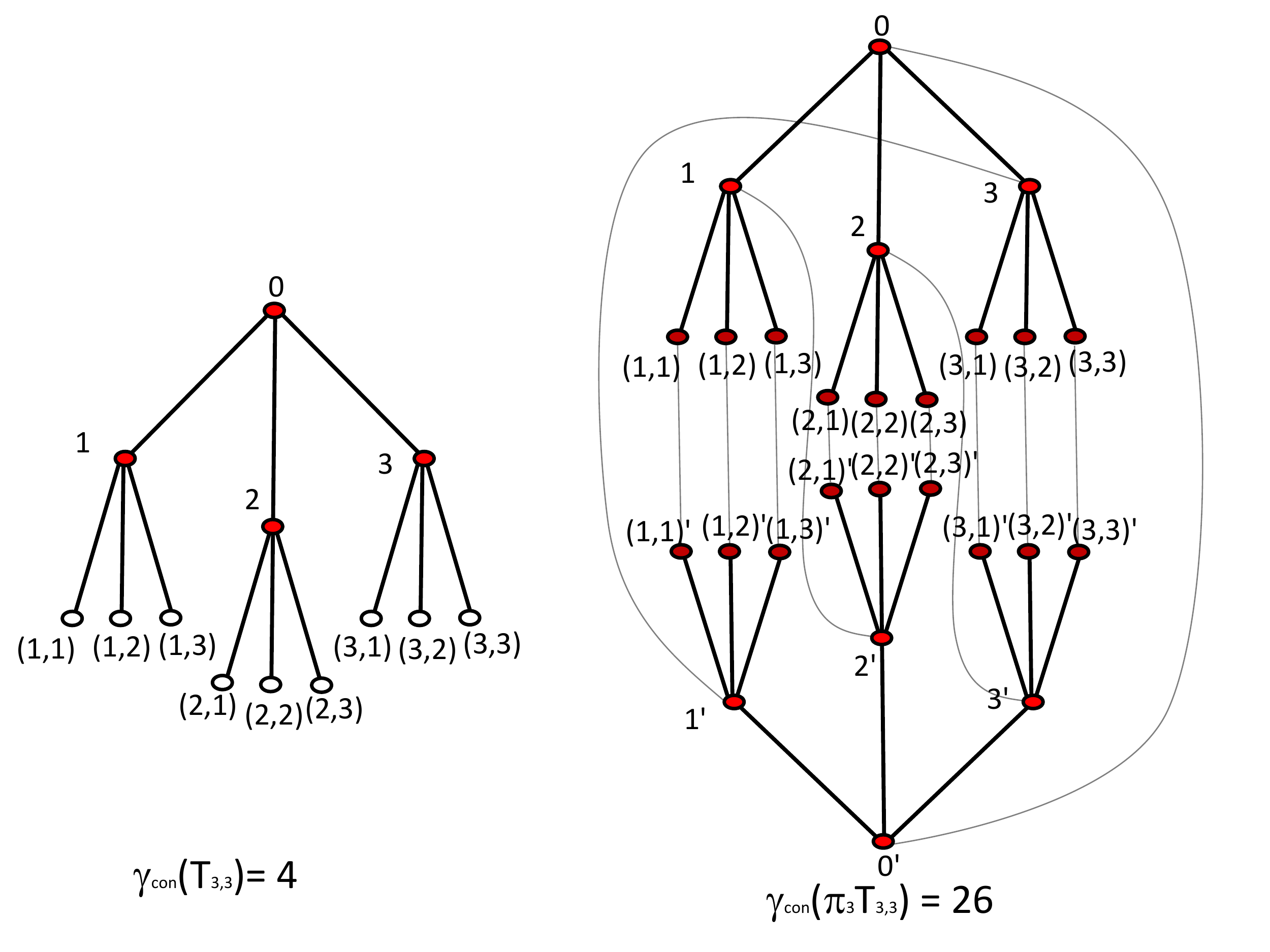}}
	\end{picture}
	\caption{The graphs $T_{3,3}$ and $\pi_{3,3} T_{3,3}$ and their $\gamma_{con}$-sets.}
	\label{fig:Tkl}
\end{figure}
\end{center}

Every convex set of $\pi_{k,l} T_{k,l}$ which dominates $S=\{(i,j):1\leq i\leq k, 1\leq j\leq l\}\cup \{(i,j)':1\leq i\leq k, 1\leq j\leq l\}$ contains $\{1,...,k,1',...,k'\}$. Every convex set containing $\{1,...,k,1',...,k'\}$ also contains $S\cup\{0,0'\},$ as $i,0,0',i'$ and $i,(i,j),(i,j)',i'$ are all shortest $i-i'$ paths for $1\leq i\leq k.$ Thus, we have $\gamma_{con}(\pi_{k,l} T_{k,l})=2kl+2k+2$. At the same time we have $\gamma_{con}(T_{k,l})=k+1$. Therefore, $\gamma_{con}(\pi_{k,l} T_{k,l})-2\gamma_{con}(T_{k,l})=2kl.$

The first inequality in (\ref{eq1}) can also be violated. 
For $k\geq 3$ let $G$ be a graph constructed as follows (see Fig. \ref{fig:Gk}):

\begin{enumerate}
\item Take $k$ copies of the path $P_7$ with $V_{P_7^i}=\{(j,i): 1\leq j\leq 7\}$ and $E_{P_7^i}=\{(j,i)(j+1,i): 2\leq j\leq 6\};$
\item For $j\in\{1,2,6,7\}$ replace the set $\{(j,i): 1\leq i\leq k\}$ with a single vertex $j;$
\item For $2\leq i\leq k$ add edges $(4,1)(4,i).$
\end{enumerate}

We define the permutation $\pi_k$ as $\pi_k=(26(5,1)(3,1)).$

Every convex dominating set of $G_k$ must contain the vertices $2$ and $6$, as well as all vertices of every shortest $2$--$6$ path. Since $2,(3,i),(4,i),(5,i),6$ for $i\in\{1,...,k\}$ are all shortest $2$--$6$ paths, $\gamma_{con}(G_k) = 3k + 2.$ However, the set $\{2,2',(3,1),(3,1)',(4,1),$ $(4,1)',(5,1),(5,1)',6,6'\}$ is a convex dominating set of cardinality $10$ in $\pi_kG_k$ for any $k$. Thus, the difference $\gamma_{con}(G)-\gamma_{con}(\pi G)$ can be arbitrarily large. 

\begin{center}
\begin{figure}[H]
	\begin{picture}(300,190)
	\put(20,-20){\includegraphics[scale=0.4]{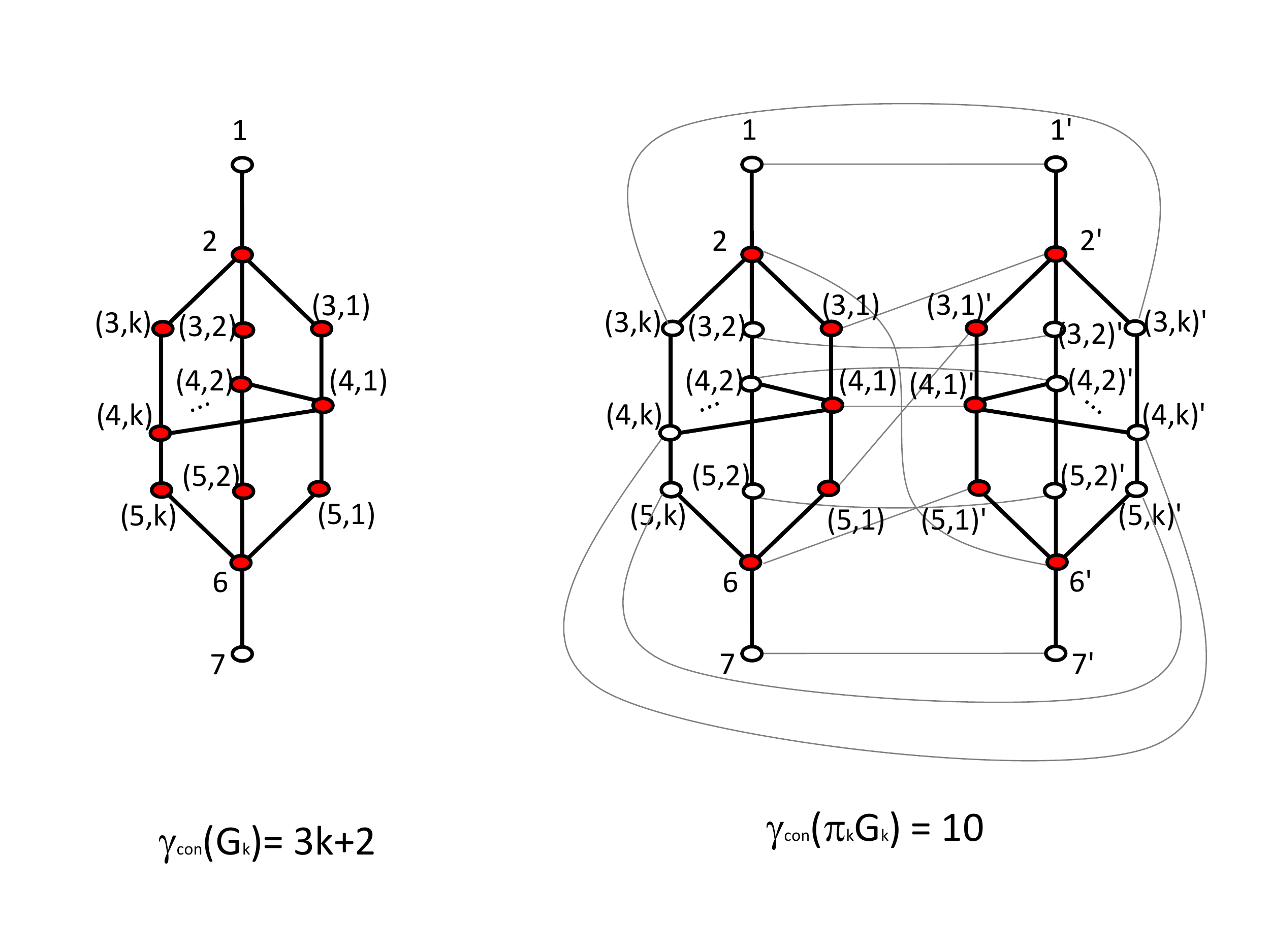}}
	\end{picture}
	\caption{The graphs $G_k$ and $\pi_kG_k$ and their $\gamma_{con}$-sets.}
	\label{fig:Gk}
\end{figure}
\end{center}

In fact, the above examples show a stronger property of the convex domination number.

\begin{rem}
The convex domination number of $\pi G$ cannot be bounded in terms of $\gamma_{con}(G).$
\end{rem}

\begin{proof}
Notice that for the graphs $G_k$ defined above $\gamma_{con}(\pi_kG_k)$ is constant, while $\gamma_{con}(G_k)$ grows with the increase of $k$. This shows that there is no upper bound on $\gamma_{con}(\pi G)$ depending only on $\gamma_{con}(G).$

Similarly, if $k$ is constant and $l$ increases, $\gamma_{con}(\pi_{k,l}T_{k,l})$ for the tree $T_{k,l}$ increases, while $\gamma_{con}(T_{k,l})$ remains constant. Thus there is no lower bound on $\gamma_{con}(G)$ depending solely on $\gamma_{con}(G)$.
\end{proof}

\end{document}